\documentclass{amsart}
\usepackage{amsthm}
\usepackage{amsmath}

\theoremstyle{plain}
\newtheorem{thm}{Theorem}
\newtheorem{lem}[thm]{Lemma}

\theoremstyle{definition}

\begin{document}

\title{The word problem for free groups cannot be solved in linear time*}
\author{Alessandro Sisto}
	\address{Maxwell Institute and Department of Mathematics, Heriot-Watt University, Edinburgh, UK}
	\email{a.sisto@hw.ac.uk}

 \begin{abstract}
  *by a standard (one-tape) Turing machine. It is well-known that the word problem for hyperbolic groups, whence in particular for free groups, can be solved in linear time. However, these algorithms run on machines more complicated than a standard Turing machine. By contrast, in this note we show that a standard Turing machine cannot solve the word problem for the free group on two generators in less than quadratic time.
 \end{abstract}
 
 \maketitle
 
 The word problem for hyperbolic groups can be solved in linear time (see, e.g. \cite[Theorem 2.28]{Alonso}), but this requires using two tapes (or more for a real-time solution \cite{Holt:real-time}). Theorem \ref{thm:quadratic} below shows that with one tape one cannot do better than quadratic time, even for a free group on two generators. All Turing machines below are one-tape.

 We will consider the free group $F_2$ on two generators $a,b$. When referring to words, we will always mean words in the alphabet $\mathcal A=\{a,b,\bar a,\bar b\}$ (where $\bar a,\bar b$ represent the inverses of $a,b$). Recall that a word in $\mathcal A$ is called trivial if it represents the trivial element of $F_2$.
 
  In the theorem below, we regard words in the alphabet $\mathcal A$ as inputs of Turing machines where we index the cells of the tapes by integers. Specifically, we assume that the word is written starting at cell $1$, which we also assume to be the starting position of the head.
 
 \begin{thm}
 \label{thm:quadratic}
  There does not exist a Turing machine with tape alphabet symbols $\mathcal A$ which accepts a word in $\mathcal A$ if and only if it is trivial, and which halts in $o(n^2)$ time when the input word has length $n$.
 \end{thm}
 
 \begin{proof}
 We will need large sets of trivial words with suitable additional properties whose use will be clear later.
 
 \begin{lem}
 \label{lem:words}
  There exists $\alpha>1,n_0\in\mathbb N$ and sets of words $\mathcal W_n$ with the following properties for all even $n$.
  
  \begin{enumerate}
   \item Each $w\in \mathcal W_n$ is trivial and has length $n$.
   \item $|\mathcal W_n|\geq \alpha^n$ for all $n\geq n_0$.
   \item\label{item:initial_words} For all integers $t\in [n/4,n/2)$, the initial subwords of length $t$ of the words in $W_n$ all represent distinct elements of $F_2$. 
  \end{enumerate}

 \end{lem}
 
 \begin{proof}
  We can take the initial words of length $n/2$ of the elements of $\mathcal W_n$ to be words in $\{a,b\}$ only, and property \eqref{item:initial_words} will hold as long as the sub-subwords of length $n/4$ are distinct. We can then ``complete'' the various words just by appending the inverse of the chosen initial words.
 \end{proof}
 Fix $\alpha>1$ and $n_0$ as in Lemma \ref{lem:words}, as well as the corresponding sets of words.
  Consider a Turing machine with $K\geq 2$ states, and suppose that it runs in $\leq \epsilon n^2$ time when the input word has length $n$, where we fix some $0<\epsilon<\log(\alpha)/(4\log(K))$. We show that the Turing machine accepts at least one non-trivial word.
  
  Given any word $w\in \mathcal W_n$ for an even $n\geq n_0$, there exists some index $c=c(w)\in [n/4,n/2)$ such that the head of the Turing machine lies at $c$ at most $4\epsilon n$ times. We call such $c$ a checkpoint index. Moreover, either for a set $\mathcal W'_n\subseteq \mathcal W_n$ containing at least half of the words in $\mathcal W_n$ the final position of the head is larger than $c(w)$ on input $w$, or it is at most $c(w)$. We assume the former for infinitely many $n$, the argument for the latter being similar (noting that item \eqref{item:initial_words} of the lemma automatically holds for final subwords as well since all words in $\mathcal W_n$ are trivial). We only consider such $n$ below, and also we assume $n\geq n_0$.
  
It is not hard to see that whether a word $w$ in $\mathcal W'_n$ is accepted by the Turing machine only then depends on:

\begin{itemize}
 \item the checkpoint index,
 \item the sequence of states at the times when it moves from $c$ to $c+1$ or vice versa (which we call checkpoint sequence),
  \item the final subword of $w$ starting after the checkpoint index.
\end{itemize}

Indeed, the operations performed by the Turing machine in each of the visits of its head beyond the checkpoint (that is, while the head is at indices larger than $c$) only depend on the state at the time it moves from $c$ to $c+1$. The reason we included information on the states when the head moves back will be clear later.

Note that there are $\leq n/4$ possible checkpoint indices. Since the checkpoint sequences have length at most $4\epsilon n$ and the Turing machine has $K$ states, there are at most
$$\frac{n}{4} 2 K^{4\epsilon n}= n K^{4\epsilon n}/2$$
possible pairs of checkpoint index and checkpoint sequence. By our choice of $\epsilon$, for any sufficiently large $n$ we have
$$n K^{4\epsilon n}/2<\alpha^n/2\leq |\mathcal W'_n|.$$
For such $n$, we must have two words $w_1,w_2$ in $\mathcal W'_n$ with the same checkpoint index $c$ and the same checkpoint sequence. We can then consider the ``hybrid'' word $w'$ that has the same initial subword of length $c$ as $w_1$ and the same final subword as $w_2$ after $c$. We have that $w'$ cannot represent the identity of $F_2$ in view of the fact that $w_2$ does, and the initial subword of length $c$ in $w'$ represents a different group element. However, $w_2$ is a accepted by the Turing machine if and only if $w'$ is, since the operations performed by the Turing machine within the checkpoint only depend on the checkpoint sequence. Therefore, the Turing machine accepts at least one non-trivial word, as required.
 \end{proof}

 \bibliographystyle{alpha}
 \bibliography{biblio}

\end{document}